\documentclass[11pt, a4paper]{article}

\usepackage[numbers]{natbib}
\usepackage{graphicx}
\usepackage{color}
\usepackage{amsmath}
\usepackage{amssymb}
\usepackage{amscd}
\usepackage{bbm}
\usepackage{tikz}
\usepackage{hyperref}
\hypersetup{
    bookmarks=true,         
    colorlinks=true,       
    linkcolor=red,          
    citecolor=blue,        
    filecolor=cyan,      
    urlcolor=cyan           
}

\usepackage[utf8]{inputenc}
\usepackage{amsthm}
\usepackage{bbm} 
\usepackage{algorithm}
\usepackage{algorithmic}

\usepackage{marginnote}

\usepackage[margin=31mm, a4paper]{geometry}

\newcommand{\inr}[1]{\bigl< #1 \bigr>}

\newcommand{\norm}[1]{\left\|#1\right\|}%

\newcommand\eps{\epsilon}

\DeclareMathOperator*{\argmin}{argmin}

\newcommand{\design}{\mathbb{X}}
\newcommand{\hbeta}{{\boldsymbol{\hat\beta}}}

\def\ds1{\textrm{1\kern-0.25emI}} 

\newcommand \E{\mathbb{E}}
\newcommand \R{\mathbb{R}}

\newcommand \cB{{\cal B}}

\newcommand \cK{{\cal K}}

\newcommand \cN{{\cal N}}

\newcommand \cP{{\cal P}}

\newcommand \bA{{\mathbb A}}
\newcommand \bB{{\mathbb B}}
\newcommand \bC{{\mathbb C}}

\newcommand \bE{{\mathbb E}}

\newcommand \bP{{\mathbb P}}

\newcommand \bR{{\mathbb R}}

\newcommand \bX{{\mathbb X}}

\newcommand{\vf}{{\mathbf{f}}}

\newcommand{\vbeta}{{\boldsymbol{\beta}}}
\newcommand{\veta}{{\boldsymbol{\eta}}}

\newcommand{\vv}{{\boldsymbol{v}}}

\newcommand{\vzero}{{\boldsymbol{0}}}
\newcommand{\vx}{{\boldsymbol{x}}}

\newcommand{\vu}{{\boldsymbol{u}}}
\newcommand{\vw}{{\boldsymbol{w}}}

\newcommand{\vy}{\mathbf{y}}
\newcommand{\vxi}{{\boldsymbol{\xi}}}

\newcommand{\vdelta}{{\boldsymbol{\delta}}}

\title{\bf Towards the study of least squares estimators with convex penalty}
\author{Pierre~C.~Bellec$^2$, Guillaume~Lecu\'e$^1$, Alexandre~B.~Tsybakov$^1$}
\date{$^1$ CREST-ENSAE, CNRS UMR9194 \\ $^2$ Rutgers University}

\usepackage{thmtools}
\declaretheorem[name=Theorem,numberwithin=section]{thm}
\declaretheorem[name=Theorem,sibling=thm]{theorem}
\declaretheorem[name=Proposition,sibling=thm]{proposition}
\declaretheorem[name=Lemma,sibling=thm]{lemma}
\declaretheorem[name=Corollary,sibling=thm]{corollary}
\declaretheorem[name=Assumption,sibling=thm]{assumption}
\declaretheorem[name=Definition,sibling=thm]{definition}

\newtheoremstyle{named}{}{}{\itshape}{}{\bfseries}{.}{.5em}{\thmnote{#3} #1}
\theoremstyle{named}

\numberwithin{equation}{section}

\begin{document}

\maketitle

\abstract{\footnotesize
Penalized least squares estimation is a popular technique in high-dimensional statistics. It  includes such methods as the LASSO, the group LASSO, and the nuclear norm penalized least squares. The existing theory of these methods is not fully satisfying since it allows one to prove oracle inequalities with fixed high probability only for the estimators depending on this probability. Furthermore, the control of compatibility factors appearing in the oracle bounds is often not explicit. Some very recent 
developments suggest that the theory of oracle inequalities can be revised in an improved way. In this paper,  we provide an overview  of ideas and tools leading to such an improved theory. We show that, along with overcoming the disadvantages mentioned above, the methodology extends to the hilbertian framework and it applies to a large class of convex penalties.  This paper is partly expository. In particular, we provide adapted proofs of some results from other recent work.

}

\section{Introduction}

Penalized least squares (LS) estimators play an important role in statistics. One of the classical examples is ridge regression estimator, for which the penalty is defined as the squared Euclidean norm. 
More recently, a  great deal of attention has been focused on high-dimensional statistical models. In this context, some new penalized LS estimators have been proposed and extensively studied. The most famous examples are the  LASSO (i.e., the $\ell_1$ norm penalized estimator) and its generalizations such as the group LASSO or the nuclear norm penalized least squares for matrix estimation. A common feature of these and related estimators is the fact that the penalty is a norm satisfying some specific decomposability conditions. Starting from \cite{bickel2009simultaneous}, there has been a considerable interest in developing a general approach to the analysis of such methods. For a detailed account, we refer the reader to \cite{koltch_book,giraud2014introduction,vdgeer} where further references can be found. As shown in~\cite{bickel2009simultaneous}, the two main ingredients of the analysis are geometric considerations based on  the restricted eigenvalue (compatibility) property, and the empirical process bounds on the stochastic error. With this approach, several techniques have been proposed for a unified treatment of LS estimators with decomposable penalties, see the overviews in \cite{negahban2012,taylor2013,vdgeer}.

 However, the existing theory is not fully satisfying in the following aspects.  
 \begin{itemize}
 \item[(i)] The results are obtained in the form of oracle inequalities depending on the restricted eigenvalue (compatibility) parameters that are, in general, not specified. An exception is the standard LASSO, for which the values of these parameters are evaluated in some situations.  
 \item[(ii)] 
     The penalties (and thus, the estimators) considered  in that theory depend on the confidence level (the probability), with which the oracle inequality holds. In other words, there is no means, in that framework, to provide oracle bounds for one given penalized LS estimator with any given confidence level.  As a consequence, oracle inequalities for the mean squared risk or {upper bounds on any other moments of the risk} are not derivable from these results. 
\end{itemize}

Very recent developments show that, in some cases, the problems (i)
and (ii) can be resolved. For (i),  a relatively general solution can be  obtained by the small ball method of \cite{MR3431642,MR3367000}.  It has been already successfully implemented for such procedures as LASSO and SLOPE  \cite{Dirksen2015OnTG,reg1,bellec2016slope}.

Techniques to overcome the disadvantage (ii) have been recently proposed in \cite{bellec2016slope,bel-tsyb}. The argument in \cite{bellec2016slope} is based on a refined bound for the stochastic error, and the results focus on the LASSO and SLOPE estimators. Thanks to these techniques, an improvement of the classical rates is obtained for the prediction and estimation errors. In \cite{bel-tsyb}, a different argument is used to resolve the problem described in (ii). The results are valid only for the prediction error but extend to other penalized LS estimators than LASSO and SLOPE. The proof is based on a Lipschitz property of the solutions and the Gaussian concentration theorem. 

  In view of these developments, the theory of oracle inequalities for penalized LS estimators can be revised in an improved way. In this paper,  we provide an overview  of ideas and tools leading to such an improved theory. Along with overcoming the disadvantages mentioned in (i) and (ii), the method extends to the hilbertian framework and applies to a large class of convex penalties.  The approach is based on a refinement of the argument in \cite{bellec2016slope}. This paper is partly expository. In particular, we provide adapted proofs of some results from the previous work.


\section{Statement of the problem} 

Assume that we observe the vector
\begin{equation*}
    \vy = \vf + \vxi,
\end{equation*}
where $\vf\in\R^n$ is an unknown deterministic mean
and $\vxi\in\R^n$ is a random noise vector. Let $\sigma>0$. We assume that $\vxi$ has normal distribution $\cN(\vzero,\sigma^2 I_{n\times n})$, where  $I_{n\times n}$ denotes the $n\times n$ identity matrix.

For all $\vu = (u_1,\dots,u_n)\in\R^n$, define the empirical norm of $\vu$
by
\begin{equation*}
    \norm{\vu}_n^2 = \frac 1 n \sum_{i=1}^n u_i^2.
\end{equation*}

Let $H$ be a Hilbert space with the inner product $\langle \cdot, \cdot \rangle$ and the corresponding norm $\Vert \cdot \Vert_H$. Let $\mathbb B$ a convex subset of $H$ such that $\mathbb B$ is a closed set
with respect to $\Vert \cdot \Vert_H$.
We study the performance of
the estimator $\hbeta$ defined as a solution of the following
minimization problem:
\begin{equation}\label{eq:def-hbeta}
    \hbeta \in \argmin_{\vbeta\in \mathbb B}
    \left(
        \norm{\design\vbeta - \vy}_n^2
        + F(\vbeta)
    \right)
\end{equation}
where $\design:H\to \R^{n}$ is a linear operator and $F:H\to \R$ is a convex function called a penalty. Our main results will be given for the case when $F(\vbeta)$ is some norm of $\vbeta$. The value $\design\hbeta$ is used as a prediction for $\vf$. If the model is well-specified, that is $\vf=\design\vbeta^*$ for some $\vbeta^*\in \mathbb B$, then
$\hbeta$ is used as an estimator of $\vbeta^*$.


\section{Basic tools}

In this section, we provide two basic facts that are used in the subsequent argument. The first of them is the following proposition generalizing \cite[Proposition E.3]{bellec2016slope} that plays a role of a ``reduction lemma'' for the stochastic error term. It is crucial to overcome the disadvantage (ii) mentioned in the Introduction.

A mapping $h:H\rightarrow [0,\infty)$ will be called {positive} homogeneous if
$h(a\vu) = a h(\vu)$ for all $a\geq0, \vu\in H$
    and $h(\vu)>0$ for $\vu\ne \vzero$.
    Denote by $\Phi(\cdot)$ the cumulative distribution function of the standard Gaussian law.

{
\begin{proposition}
    \label{prop:deviation-abstract-norm-N}
    Assume that $\vxi\sim \mathcal N(\vzero,\sigma^2 I_{n\times n})$.
    Let $h:H\rightarrow [0,+\infty)$ be a {positive} homogeneous mapping and let $\tau>0$.    Assume that the event
    \begin{equation*}
        \Omega \triangleq \left\{
                \sup_{\vv\in H: h(\vv) {\le} 1} \frac 1 n  \vxi^T\design\vv
                \le \tau
        \right\}
    \end{equation*}
    satisfies $\mathbb P(\Omega) \ge 1/2$.
    Then for all $\delta\in(0,1)$ we have
    \begin{equation*}
        \mathbb P\left(
            \forall \vu\in H:
            \quad
        \frac 1 n \vxi^T\design\vu
        \le (\tau+1) \max\left(h(\vu), \sigma \frac{\Phi^{-1}(1-\delta)}{\sqrt{n}}  \norm{\design\vu}_n \right)
        \right)
        \ge 1 - \delta.
        \label{eq:new-deviation-abstract-norm-N}
    \end{equation*}
\end{proposition}
\begin{proof}
    By homogeneity, it is enough to consider only $\vu\in H$ such that
    $$\max\left(h(\vu), \norm{\design\vu}_n / L \right) = 1$$
    \text{where}
    $L \triangleq \sqrt{n}/ (\sigma\Phi^{-1}(1-\delta))$.
    Define $T\subset H$ and $f:\R^n\rightarrow \R$ by
    \begin{equation}
        T \triangleq \left\{\vu\in H: \max\left(h(\vu), \frac 1 L \norm{\design\vu}_n\right)   \le 1\right\},
        \qquad
        f(\vv) \triangleq
        \sup_{\vu\in T}
        \frac 1 n (\sigma \vv)^T\design\vu
        \label{eq:def-T-f}
    \end{equation}
    for all $\vv\in\R^n$.
  Then, for every $\vv_1,\vv_2\in\R^n$, $|f(\vv_1) - f(\vv_2)|\leq \left((\sigma L)/\sqrt{n}\right)|\vv_1-\vv_2|_2$ where $|\cdot|_2$ denotes the Euclidean norm onto $\R^n$. Therefore, $f$ is a Lipschitz function with Lipschitz constant $\sigma L/\sqrt n$ and by the Gaussian concentration inequality, cf., e.g., \cite[Theorem 6.2]{lifshits}, we have
   that with probability at least $1-\delta$,
    \begin{align*}
        \sup_{\vu\in T}
        \frac 1 n \vxi^T\design\vu
        &\le
        {\rm Med}
        \left[
                \sup_{\vu\in T}
                \frac 1 n \vxi^T\design\vu
        \right]
        + \sigma L \frac{\Phi^{-1}(1-\delta)}{\sqrt{n}}
        \\
       &
       { \le
         {\rm Med} \left[
                \sup_{\vu\in\R^p: h(\vu) {\le} 1}
                \frac 1 n \vxi^T\design\vu
        \right]
        + \sigma L  \frac{\Phi^{-1}(1-\delta)}{\sqrt{n}}
        }
        \\
        &\le
        \tau
        + \sigma L  \frac{\Phi^{-1}(1-\delta)}{\sqrt{n}} = \tau+1,
    \end{align*}
    where  ${\rm Med} [\zeta]$ denotes the median of random variable $\zeta$ and we have used the fact that $\mathbb P(\Omega)\ge 1/2$ to bound the median from above.
\end{proof}


The next proposition is a simple fact based on convexity argument. In different versions, it was used as an element of the proof of oracle inequalities with leading constant 1 starting from \cite{koltchinskii2011nuclear}. Some special cases of it are explicitly stated in \cite[Lemma 1]{bellec2016prediction}
    and \cite[Lemma A.2]{bellec2016slope}.

\begin{proposition}\label{prop2}
    Let $F:H\rightarrow\R$ be a convex function, and
    let $\design:H\to \R^n$ be a linear operator.
    If $\hbeta$ is a solution of the minimization problem \eqref{eq:def-hbeta},
    then $\hbeta$ satisfies, for all $\vbeta\in\mathbb B$ and all $\vf\in\R^n$,
    \begin{equation}
        \|\design\hbeta - \vf\|^2_n
        -
        \norm{\design\vbeta - \vf}^2_n
        \le
        \frac 2 n \vxi^T (\design(\hbeta - \vbeta)) 
        + F(\vbeta)
        - F(\hbeta)
        - \|\design(\hbeta - \vbeta)\|^2_n.
        \label{eq:almost-sure}
    \end{equation}
\end{proposition}
\begin{proof}
    Define the functions $f$ and $g$ by the relations $g(\vbeta) = \norm{\design\vbeta - \vy}^2_n$,
    and $f(\vbeta) = g(\vbeta) + F(\vbeta)$ for all $\vbeta\in H$.
    Since $f$ is convex and $\hbeta$ is a minimizer of $f$ on $\mathbb B$, it follows that for some $\vw$ in the sub-differential of $f$ at $\hbeta$, we have
    $\langle \hbeta-\vbeta,\vw\rangle \le 0$ for all $\vbeta\in \mathbb B$,  cf., e.g., \cite{peypouquet2015convex}.
    Using the Moreau-Rockafellar theorem, we obtain that there exists $\vv$ in the sub-differential of $F$ at $\hbeta$ such that
    $ \langle \hbeta-\vbeta,\vw\rangle =
         \langle \hbeta-\vbeta,\frac2n \design^*(\design \hbeta - \vy)+\vv\rangle
        $ for all $\vbeta\in \bB$
   where $\design^*$ is the adjoint operator of $\design$. Thus,
    \begin{align*}
        \frac{2}{n} (\design(\hbeta - \vbeta))^T (\design \hbeta - \vy)
         \le
        \langle \vbeta - \hbeta, \vv\rangle.
    \end{align*}
   Note also that by simple algebra,
   \begin{align*}
        &\|\design\hbeta - \vf\|^2_n
        -
        \norm{\design\vbeta - \vf}^2_n
        =
        \frac{2}{n} (\design(\hbeta - \vbeta))^T (\design \hbeta - \vf)
        -
        \|\design(\hbeta - \vbeta)\|^2_n.
    \end{align*}
    Combining the last two displays we obtain
    \begin{align*}
        &\|\design\hbeta - \vf\|^2_n
        -
        \norm{\design\vbeta - \vf}^2_n
        \le
        \frac 2 n \vxi^T (\design(\hbeta - \vbeta))
        -
        \|\design(\hbeta - \vbeta)\|^2_n
        +  \langle\vbeta-\hbeta,\vv\rangle.
    \end{align*}
    To complete the proof, notice that by definition of the subdifferential of $F$ at $\hbeta$,
    we have
      $  \langle\vbeta-\hbeta,\vv\rangle \le F(\vbeta) - F(\hbeta).$
\end{proof}

\section{Oracle inequalities}

In this section, we consider a Hilbert space $H$ and a linear operator $\mathbb X:H\to \R^n$ defined by the relation
$$
\mathbb X \vbeta  = ( \langle \vbeta, X_1\rangle, \dots,  \langle \vbeta, X_n\rangle)^\top, \quad \forall \vbeta\in H,
$$
where $X_1, \ldots, X_n$ are deterministic elements of $H$.

We will also assume that $F(\vbeta)=\lambda\|\vbeta\|$ where $\|\cdot\|$ is a norm on $H$ (called the regularization norm) and $\lambda>0$ is a tuning constant.
Thus, the minimization problem \eqref{eq:def-hbeta} takes the form
\begin{equation}\label{eq:def-hbeta_1}
    \hbeta \in \argmin_{\vbeta\in \mathbb B}
    \left(
        \norm{\design\vbeta - \vy}_n^2
        + \lambda\|\vbeta\|
    \right)
\end{equation}
where $\mathbb B$ is a closed convex subset of $H$.

To each matrix $A\in H$, we associate a linear operator
$\cP_A: H\to H$. Examples of $\cP_A$ that are interesting in the context of high-dimensional statistics will be given later.
Set $\cP_A^{\perp}= I-\cP_A$ where $I$ is the identity operator on $H$.
The following assumption will be crucial for the subsequent argument.
\begin{assumption}\label{assum:decomp}
 There exists a subset $\mathbb A$ of $\mathbb B$ such that
\begin{equation*}\label{eq:assumption1_1}
{\cal P}_AA= A, \qquad \forall \ A\in \mathbb A,
\end{equation*}
\begin{equation}\label{eq:assumption1_2}
 \|A\| - \|B \|
    \le
    \|\mathcal P_A(A - B) \| - \|{\cal P}^\perp_A B \|
    ,
        \qquad \forall \ A\in \mathbb A, \ \forall \ B\in H.
\end{equation}
\end{assumption}
\noindent
Note that since $\mathcal P_A A = A$, inequality \eqref{eq:assumption1_2} can be rewritten as
\begin{equation}\label{eq:assumption1_3}
 \|A\| + \|{\cal P}^\perp_A B \|
    \le
    \|A - \mathcal P_A(B) \| + \| B \|
    ,
        \qquad \forall \ A\in \mathbb A, \ \forall \ B\in H.
\end{equation}
Looking at \eqref{eq:assumption1_3}, it is easy to check that Assumption \ref{assum:decomp} is satisfied if the following decomposability property holds. 
\begin{assumption}[Decomposability assumption]\label{assum:decomp2}
 There exists a subset $\mathbb A$ of $\mathbb B$ such that
\begin{equation*}\label{eq:assumption2_1}
{\cal P}_AA= A, \qquad \forall \ A\in \mathbb A,
\end{equation*}
\begin{equation}
    \label{eq:assumption2_2}
     \|A\| + \|{\cal P}^\perp_A B \|
    =
    \|A + {\cal P}^\perp_A B \| ,
    \qquad \forall \ A\in \mathbb A, \ \forall \ B\in H.  
\end{equation}
\end{assumption}
This decomposability assumption is satisfied, with suitable definitions of ${\cal P}_A$, 
for the three regularization norms $\|\cdot\|$ playing the central role in high-dimensional statistics: 
the $\ell_1$-norm, the group LASSO norm, and the nuclear norm.  They are analyzed in Section~\ref{sec:examples}.

Beyond the decomposable case, one may turn to other assumptions stated in terms of the ``size'' of sub-differential of the regularization norm, cf. \cite{reg1,luu2017sharp,vaiter2015model}.
The articles \cite{luu2017sharp,vaiter2015model} propose the following methodology to define projectors $\mathcal P_A$ and $\mathcal P_A^\perp$ for any norm $\|\cdot\|$.
Given a finite dimensional Hilbert space $H$ and a matrix $A\in H$, 
the subdifferential $\partial \|\cdot\|(A)$ of $\|\cdot\|$ at $A$ is a convex subset of $H$.
The set $\partial \|\cdot\|(A)$ is endowed with a unique affine envelope, i.e., the smallest affine subspace of $H$ that contains $\partial \|\cdot\|(A)$.
The affine envelope of $\partial \|\cdot\|(A)$ is a subset of $H$ of the form $e_A+V_A$ where $e_A\in H$ and $V_A$ is a linear subspace of $H$ ($V_A$ is sometimes referred to as the direction of the affine subspace).
Finally, the projector $\mathcal P_A^\perp$ is defined as the orthogonal projection onto $V_A$, and $\mathcal P_A$ is defined as the orthogonal projection onto $V_A^\perp$.
It can be shown that $\mathcal P_A(A) = A$ for any $A\in H$ and that the above definition of $\mathcal P_A$ and $\mathcal P_A^\perp$ yields the same projectors as those studied
in Section~\ref{sec:examples} for the Lasso, the Group-Lasso and the Nuclear norm estimators \cite{luu2017sharp,vaiter2015model}.
We refer the reader to \cite{luu2017sharp,vaiter2015model} for more details
on this general definition of $\mathcal P_A$ and $\mathcal P_A^\perp$.

To state the result, we will need some notation. For any $A\in H$ and any  constant $c_0>0$, define the following cone in ${\mathbb B}$:
$$
{\mathbb C}_{A,c_0}\triangleq \Bigl\{B \in {\mathbb B}: \|{\cal
P}_A^{\perp}B\|\leq c_0 \|{\cal P}_AB\|\Bigr\},
$$
and introduce the associated quantity that we will call the {\it compatibility factor}~:
\begin{equation}
\label{eq:re_constant}
    \mu_{c_0}(A)\triangleq  \inf\Bigl\{\mu'>0: \|{\cal P}_A B\| \leq \mu'
    \|{\mathbb X}B\|_{n}, \,\forall B\in {\mathbb C}_{A,c_0}\Bigr\}.
\end{equation}
Note that $\mu_{c_0}(A)$ is a nondecreasing function of $c_0 .$
\begin{theorem}
    \label{th1}
    Assume that $\vxi\sim \mathcal N(\vzero,\sigma^2 I_{n\times n})$, and that Assumption \ref{assum:decomp} holds. 
     Let $\tau'>0$ be such that the event
    \begin{equation*}
        \Omega = \left\{
                \sup_{\vv\in H: \|\vv\| {\le} 1} \frac 1 n  \vxi^T\design\vv
                \le \tau'
        \right\}
    \end{equation*}
    satisfies $\mathbb P(\Omega) \ge 1/2$. Let $\lambda\ge10\tau'$ and $\delta\in (0,1)$. Then, the estimator $\hbeta$ defined in \eqref{eq:def-hbeta_1} satisfies, with probability at least $1-\delta$,
    \begin{equation}
        \|
            \design\hbeta
            -
            \vf
        \|_n^2
        \le
        \inf_{\vbeta \in {\mathbb A} }
               \Big[
             \|
                \design\vbeta
                -
                \vf
            \|_n^2
            +\frac{16}{25}\lambda^2 \mu_4^2 (\vbeta)
        \Big] + \frac{16\sigma^2(\Phi^{-1}(1-\delta))^2}{n}
        \label{eq:soi-squared-lasso}
    \end{equation}
    where, in particular, $(\Phi^{-1}(1-\delta))^2\le 2\log(1/\delta)$.
    If, in addition,
    $\vf =\design\vbeta^* $ for some $\vbeta^* \in \mathbb A$, then with probability at least $1-\delta$,
    \begin{equation}
        \|
            \hbeta
            -
            \vbeta^*
        \|
        \le
              4\lambda \mu_4^2 (\vbeta^*)
         + \frac{20\sigma^2 (\Phi^{-1}(1-\delta))^2}{n \lambda}.
        \label{eq:soi-squared-lasso-2}
\end{equation}
\end{theorem}
\begin{proof}
 Note that 
    \begin{equation*}
        \Omega = \left\{
                \sup_{\vv\in H: \lambda \|\vv\|/5 {\le} 1} \frac 1 n  \vxi^T\design\vv
                \le 5\tau'/\lambda
        \right\}.
    \end{equation*}
    By Proposition~\ref{prop:deviation-abstract-norm-N} with $h(\vv)=\lambda \|\vv\|/5$ and $\tau=5\tau'/\lambda$ we obtain
that, on an event $\Omega'$ of probability at least $1-\delta$,
    \begin{equation*}
            \forall \vu\in H:
            \quad
        \frac 1 n \vxi^T\design\vu
        \le (5\tau'/\lambda+1) \max\left(\lambda \|\vu\|/5, \nu \norm{\design\vu}_n \right)
            \end{equation*}
where      
$$
\nu =  \frac{\sigma\Phi^{-1}(1-\delta)}{\sqrt{n}} .
$$
In the rest of the proof, we will place ourselves on the event $\Omega'$. Using Proposition~\ref{prop2} and the last display we find that on $\Omega'$, for all $\vbeta\in \mathbb B$,
\begin{align}\nonumber
\|\design\hbeta - \vf\|^2_n-\norm{\design\vbeta - \vf}^2_n
&\le 2(5\tau'/\lambda+1) \max\left(\lambda \|\vu\|/5,  \nu  \norm{\design\vu}_n \right)  \\
& \nonumber \qquad \quad + \lambda\|\vbeta\|- \lambda\|\hbeta\|
- \norm{\design \vu}^2_n     \\  
& \le 3 \max\left(\lambda \|\vu\|/5,  \nu \norm{\design\vu}_n \right) 
\label{eq:th1:1} \\
& \nonumber \qquad \quad + \lambda\|\vbeta\|
        - \lambda\|\hbeta\|
        - \norm{\design \vu}^2_n
    \end{align}
where $\vu=\hbeta - \vbeta$. We now consider separately three cases.

{\it  Case 1: Matrix $\vbeta\in \mathbb A$ is such that $\lambda \|\vu\|/5 \le \nu \norm{\design\vu}_n$.} Then,
\begin{equation}
        \|\design\hbeta - \vf\|^2_n
        -
        \norm{\design\vbeta - \vf}^2_n
               \le
  8\nu\norm{\design\vu}_n    
               - \norm{\design \vu}^2_n\le 16\nu^2.
        \label{eq:th1:2}
    \end{equation}
Thus, for such $\vbeta$ inequality \eqref{eq:soi-squared-lasso} is satisfied.

The next two cases correspond to $\vbeta\in \mathbb A$ such that $\lambda \|\vu\|/5 > \nu \norm{\design\vu}_n$. If this inequality holds, then \eqref{eq:th1:1} implies
\begin{equation}
        \|\design\hbeta - \vf\|^2_n
        -
        \norm{\design\vbeta - \vf}^2_n
              \le
   \lambda (3\|\hbeta-\vbeta\|/5+ \|\vbeta\|
        - \|\hbeta\| )
        - \norm{\design \vu}^2_n.
        \label{eq:th1:3}
    \end{equation}
Assumption~\ref{assum:decomp} with $A=\vbeta$ and $B=\hbeta$ grants that
$$\|\vbeta\| - \|\hbeta\| \le \|\cP_\vbeta(\vbeta - \hbeta)\| - \|\cP_\vbeta^\perp \hbeta\|$$
while, by the triangle inequality,
$$\|\hbeta-\vbeta\| \le \|\cP_\vbeta(\vbeta - \hbeta)\| + \|\cP_\vbeta^\perp(\vbeta-\hbeta) \|=\|\cP_\vbeta(\vbeta - \hbeta)\| + \|\cP_\vbeta^\perp\hbeta \|.$$
Combining the last two inequalities we obtain
\begin{align*}
3\|\hbeta-\vbeta\|/5+ \|\vbeta\|
        - \|\hbeta\|
\le 8\|\cP_{\vbeta}(\hbeta-\vbeta)\|/5 - 2\|\cP_{\vbeta}^{\perp}\hbeta\|/5.
    \end{align*}
This inequality and \eqref{eq:th1:3} imply
\begin{equation}
        \|\design\hbeta - \vf\|^2_n
        -
        \norm{\design\vbeta - \vf}^2_n
              \le
   (2\lambda/5)\left(4\|\cP_{\vbeta}\vu\| - \|\cP_{\vbeta}^{\perp}\vu\|\right)
        - \norm{\design \vu}^2_n.
        \label{eq:th1:5}
    \end{equation}

{\it  Case 2: Matrix $\vbeta\in \mathbb A$ is such that $\lambda \|\vu\|/5 > \nu \norm{\design\vu}_n$ and 
$4\|\cP_{\vbeta}\vu\| < \|\cP_{\vbeta}^{\perp}\vu\|$.} Then, in view of \eqref{eq:th1:5}, inequality \eqref{eq:soi-squared-lasso} holds trivially.

{\it  Case 3: Matrix $\vbeta\in \mathbb A$ is such that $\lambda \|\vu\|/5 > \nu \norm{\design\vu}_n$ and 
$4\|\cP_{\vbeta}\vu\| \ge \|\cP_{\vbeta}^{\perp}\vu\|$.} Then $\vu$ belongs to the cone ${\mathbb C}_{\vbeta,4}$, so that 
$\|\cP_{\vbeta}\vu\| \le \mu_4(\vbeta)\norm{\design \vu}_n$. This and \eqref{eq:th1:5} yield
\begin{equation*}
        \|\design\hbeta - \vf\|^2_n
        -
        \norm{\design\vbeta - \vf}^2_n
              \le
   \frac{8\lambda \mu_4(\vbeta)}{5}\norm{\design \vu}_n
        - \norm{\design \vu}^2_n \le \frac{16}{25}\lambda^2\mu_4^2(\vbeta),
        \label{eq:th1:6}
    \end{equation*}
and hence inequality \eqref{eq:soi-squared-lasso}.

Consider now the well-specified case: $\vf =\design\vbeta^* $ for some $\vbeta^* \in \mathbb A$. Set $\vu=\hbeta-\vbeta^*$. Again, we proceed in cases.

{\it  Case 1: Matrix $\vbeta^*\in \mathbb A$ is such that $\lambda \|\vu\|/5 \le \nu \norm{\design\vu}_n$.} Then, inequality \eqref{eq:th1:2} with $\vbeta=\vbeta^*$ implies $\norm{\design\vu}_n\le 4\nu$, so that $ \|\vu\|\le 20\nu^2/\lambda$. The bound \eqref{eq:soi-squared-lasso-2} follows.

{\it  Case 2: Matrix $\vbeta^*\in \mathbb A$ is such that $\lambda \|\vu\|/5 > \nu \norm{\design\vu}_n$.} Then, from \eqref{eq:th1:5} with  $\vbeta=\vbeta^*$ we obtain that $4\|\cP_{\vbeta^*}\vu\| \ge \|\cP_{\vbeta^*}^{\perp}\vu\|$, and consequently,   $\|\cP_{\vbeta^*}\vu\| \le \mu_4(\vbeta^*)\norm{\design \vu}_n$. On the other hand, \eqref{eq:th1:5} also implies that
\begin{align*}
        \norm{\design\vu}^2_n
         \le
   4\lambda\|\cP_{\vbeta^*}\vu\|/5.       
    \end{align*}
In conclusion, $\|\cP_{\vbeta^*}\vu\| \le  4\lambda\mu_4^2(\vbeta^*)/5$. Finally, $\|\vu\| =  \|\cP_{\vbeta^*}\vu\| +\|\cP_{\vbeta^*}^{\perp}\vu\| \le  5\|\cP_{\vbeta^*}\vu\| \le 4\lambda\mu_4^2(\vbeta^*)$. The bound \eqref{eq:soi-squared-lasso-2} follows.

\end{proof}

By integration over $\delta$, we can readily derive from Theorem~\ref{th1}  bounds for any moments of $\|
            \design\hbeta
            -
            \vf
        \|_n$ and $\|
            \hbeta
            -
            \vbeta^*
        \|$. In particular, the following corollary is immediate.
\begin{corollary}
    \label{cor1}
Under the assumptions of Theorem~\ref{th1}, the estimator $\hbeta$ defined in \eqref{eq:def-hbeta_1} satisfies    
\begin{equation}\label{eq:soi-squared-lasso-3}
\E \|\design\hbeta-\vf\|_n^2\le\min_{\vbeta \in {\mathbb A} }\Big[\|\design\vbeta-\vf\|_n^2+\frac{16}{25}\lambda^2 \mu_4^2 (\vbeta)\Big] + \frac{16\sigma^2}{n}.
\end{equation} 
    If, in addition, $\vf =\design\vbeta^* $ for some $\vbeta^* \in \mathbb A$, then 
\begin{equation}\label{eq:soi-squared-lasso-4}
\E \|\hbeta-\vbeta^*\|\le8\lambda \mu_4^2 (\vbeta^*)+ \frac{20\sigma }{\lambda n}.
\end{equation}
\end{corollary}
Note that the regularization parameter $\lambda$ does not depend on the parameter $\delta$ that defines the confidence level. This is a key to get results in expectation as in Corollary~\ref{cor1}.

\section{Control of the compatibility factor}
\label{sec:control_re_constant}
As follows from Theorem~\ref{th1} and Corollary~\ref{cor1}, the performance of penalized LS estimators is driven by the  compatibility factor $\mu_{c_0}(A)$ defined in \eqref{eq:re_constant}. The aim of this section is to provide a control of this quantity uniformly over all $A\in\bA$ with high probability when $X_1, \ldots, X_n$ are $n$ independent and identically distributed (i.i.d.) realizations of an $H$-valued random variable $X$. We will consider $X$ satisfying the following assumption, cf. \cite{MR3431642,MR3367000}.

\begin{assumption}[Small ball assumption]\label{ass:small_ball}
There exist constants $\beta_0>0$ and $\kappa_0\in(0,1)$ such that for all $B\in\bB$, 
\begin{equation*}
\bP\left[|\inr{X, B}|\geq \beta_0 \|B\|_H\right]\geq \kappa_0.
\end{equation*}
\end{assumption}

This assumption is rather mild. We refer the reader to \cite{MR3431642,MR3364699,MR3367000} for some examples.
A simple sufficient condition for the small ball assumption is given in the next lemma.
\begin{lemma}
    \label{lemma:moments2-4}
    Assume that $X$ is isotropic in the sense that 
    \begin{equation}
    \label{isotropic}
        \forall B \in H,
        \qquad \E\inr{X,B}^2 = \|B\|^2_H.
    \end{equation}
    Furthermore, assume that there exists a constant $L>0$ such that for any $B\in H$, 
    \begin{equation}
        \E\left[\inr{X,B}^4\right]^{1/4} \le 2 L \E\left[\inr{X,B}^2\right]^{1/2}. 
        \label{assum:moments2-4}
    \end{equation}
    Then $X$ satisfies the small ball assumption with parameters
    \begin{equation}
    \label{sb-params}
        \beta_0 = 1/\sqrt 2
        \qquad
        \text{ and }
        \qquad
        \kappa_0=1/(64L^4).
    \end{equation}
\end{lemma}
\begin{proof}
It follows from the Paley-Zygmund inequality (cf., for instance, Proposition~3.3.1 in \cite{MR1666908}) ) that   
\begin{align*}
    \bP\left(|\inr{X,  B}|\geq \beta_0 \|B\|_H\right) 
    &= \bP\left(|\inr{X,  B}|^2\geq \beta_0^2 \E\left[\inr{X,B}^2\right]\right), 
    \\
    &\geq (1-\beta_0^2)^2 \E\left[\inr{X,B}^2\right]^{2}\E\left[\inr{X,B}^4\right]^{-1}\geq (1-\beta_0^2)^2\left(\frac{1}{2L}\right)^4.
\end{align*}
Hence, $X$ satisfies the small ball assumption with parameters $\beta_0,\kappa_0$ defined in \eqref{sb-params}.
\end{proof}

Lemma \eqref{lemma:moments2-4} shows that the small ball assumption is satisfied under weak moment conditions.
Indeed, the existence of moments $\E\inr{X,B}^p$ for $p>4$ is not required.

The small ball assumption is helpful in situations where one needs  to bound from below an empirical process with nonnegative terms. Note that $\norm{\bX B}_n^2 = (1/n)\sum_{i=1}^n \inr{X_i, B}^2$ is an empirical process with nonnegative terms considered as a function of $B\in \bC_{A, c_0}$. If we obtain a uniform lower bound on it, an upper bound on the compatibility factor $\mu_{c_0}(A)$ follows. The next theorem, cf. \cite{MR3431642},  provides such a lower bound on $\norm{\bX B}_n^2$ based on the small ball argument. For the sake of completeness, we recall here its proof. 

\begin{theorem}[cf. Theorem~2.1 in \cite{MR3431642}]
\label{theo:small_ball_re}
Let $X$ be an $H$-valued random variable satisfying Assumption \ref{ass:small_ball} with parameters $\beta_0>0$ and $\kappa_0\in(0,1)$. Let $X_1, \ldots, X_n$ be $n$ i.i.d. realizations of  $X$. Assume that 
\begin{equation}\label{eq:minimal_n_re}
   \E \sup_{B\in S_2 \cap \left(\cup_{A\in\bA} \bC_{A, c_0}\right)}\left|\frac{1}{n}\sum_{i=1}^n\eps_i \inr{X_i, B}\right|\leq \frac{\kappa_0 \beta_0}{16}
\end{equation}
where $S_2$ is the unit sphere in $H$ and $\eps_1, \ldots, \eps_n$ are i.i.d. random variables uniformly distributed on $\{-1, 1\}$ and independent of $X_1, \ldots, X_n$. Then, with probability greater than $1-\exp\left(-n\kappa_0^2/32\right)$, for all $B\in \cup_{A\in\bA} \bC_{A, c_0}$ we have
\begin{equation*}
    \|\design B\|_n \geq \|B\|_H\sqrt{\frac{\beta_0^2\kappa_0}{8}}.
\end{equation*}
\end{theorem}

\begin{proof}By homogeneity, it is enough to prove the result for all $B\in \cB$ where $\cB \triangleq S_2 \cap \left(\cup_{A\in\bA} \bC_{A, c_0}\right)$. 
Denote by $P_n$ the empirical measure associated to $X_1, \ldots, X_n$. Let $B\in S_2$. We have
\begin{align}\label{eq:markov_empiric}
\nonumber &\|\design B\|_n^2 = \frac{1}{n}\sum_{i=1}^n \inr{X_i, B}^2 \triangleq P_n \inr{\cdot, B}^2\geq \frac{\beta_0^2}{4} P_n\left[|\inr{\cdot, B}|\geq (\beta_0/2)\right]\\
\nonumber & = \frac{\beta_0^2 }{4}\left\{\bP\left[|\inr{X, B}|\geq \beta_0\right] +  P_n\left[|\inr{\cdot, B}|\geq (\beta_0/2)\right] - \bP\left[|\inr{X, B}|\geq \beta_0\right]\right\}\\
&\geq \frac{\beta_0^2}{4}\left\{\kappa_0 + (P_n-P)\phi\left(|\inr{\cdot, B}|\right)\right\}
\end{align}where in the last inequality we used the small ball assumption and the fact that $P_n\left[|\inr{\cdot, B}|\geq (\beta_0/2)\right] \geq P_n\phi\left(|\inr{\cdot, B}|\right)$ and $\bP\left[|\inr{X, B}|\geq \beta_0\right]\leq P\phi\left(|\inr{\cdot, B}|\right)$ where $\phi$ is defined by
\begin{equation*}
 \phi(t) = \left\{
\begin{array}{cc}
1 & \mbox{ if } t\geq \beta_0\\
2t/\beta_0-1 & \mbox{ if } \beta_0/2\leq t \leq \beta_0\\
0  & \mbox{ otherwise.}
\end{array}
 \right.
 \end{equation*} 
Set now 
\begin{equation*}
f(X_1, \ldots, X_n) = \sup_{B\in \cB} (P-P_n)\phi\left(|\inr{\cdot, B}|\right).
\end{equation*}It follows from the bounded difference inequality (cf., for instance, Theorem~6.2 in \cite{MR3185193}) 
that for all $x>0$, with probability greater than $1-\exp(-x)$, 
\begin{equation*}\label{}
f(X_1, \ldots, X_n)\leq \E f(X_1, \ldots, X_n) + \sqrt{\frac{2x}{n}}.
\end{equation*}
This and
 the Gin{\'e}-Zinn symmetrization inequality (cf., for instance, Chapter~2.3 in \cite{MR1385671}) yields that for all $x>0$, with probability greater than $1-\exp(-x)$, 
\begin{equation}\label{eq:bounded_diff_ineq}
f(X_1, \ldots, X_n)\leq 2\E \sup_{B\in \cB} \frac{1}{n}\sum_{i=1}^n \eps_i \phi\left(|\inr{ X_i, B}|\right)+ \sqrt{\frac{2x}{n}}.
\end{equation}Note that $\phi$ is a Lipschitz function with Lipschitz constant $2/\beta_0$ and $\phi(0)=0$. Thus, it follows from the contraction inequality (cf. equation~(4.20) in \cite{MR2814399}) that
\begin{equation*}
\E \sup_{B\in \cB} \frac{1}{n}\sum_{i=1}^n \eps_i \phi\left(|\inr{X_i, B}|\right)\leq \frac{2}{\beta_0}\E \sup_{B\in \cB} \frac{1}{n}\sum_{i=1}^n \eps_i \inr{X_i, B}\leq \frac{\kappa_0}{8}
\end{equation*}
where the last inequality is due to  \eqref{eq:minimal_n_re}.
Combining this bound with \eqref{eq:bounded_diff_ineq} and choosing $x=n \kappa_0^2/32$ we obtain that $f(X_1, \ldots, X_n)\leq \kappa_0/2$ with probability greater than $1-\exp\left(-n \kappa_0^2/32\right)$. Therefore, with the same probability, $(P_n-P)\phi\left(|\inr{\cdot, B}|\right)\geq -\kappa_0/2$ for all $B\in\cB$. This and \eqref{eq:markov_empiric} prove the theorem.
\end{proof}

It follows from Theorem~\ref{theo:small_ball_re} that if $X$ satisfies the small ball assumption and $n$ is large enough so that \eqref{eq:minimal_n_re} holds then, with probability greater than $1-\exp(-n\kappa_0^2/32)$, for all $A\in\bA$, 
\begin{equation}\label{eq:condi_mu}
    \mu_{c_0}(A)\leq \left(\frac{8}{\beta_0^2\kappa_0}\right)^{1/2}\sup_{B\in\bC_{A, c_0}}\frac{\norm{\cP_A B}}{\|B\|_H}. 
\end{equation}
Thus, we have reduced the control of $\mu_{c_0}(A)$ to the bound \eqref{eq:minimal_n_re} on the expectation of the empirical process.  Under certain assumptions, this expectation can be controlled in terms of the \textit{Gaussian mean width} of the set $S_2 \cap \left(\cup_{A\in\bA} \bC_{A, c_0}\right)$ as explained below. Then, we can derive an estimate on
a sufficient number $n$ of observations for \eqref{eq:minimal_n_re} to hold.
The argument can be carried out using the main result from
\cite{shahar_expect}.
To state this result, we first introduce the definition of the Gaussian mean width of a subset of a Hilbert space
and the definition of a $K$-unconditional norm.

Let $\bC$ be a subset of the Hilbert space $H$. We denote by $(G_B)_{B\in\bC}$ the centered gaussian process indexed by $\bC$ having the same covariance structure as $X$, that is $\E G_B=0$ and $\E G_{B_1}G_{B_2}=\E\inr{B_1,X}\inr{X,B_2}$ for all $B,B_1,B_2\in\bC$ (we refer the reader to \cite{lifshits} or to Chapter~12 in \cite{MR1932358} for more details on Gaussian processes in Hilbert spaces). The \textit{Gaussian mean width of $\bC$} is defined as
\begin{equation}\label{eq:gauss_width}
\ell^*(\bC) = \sup\left\{\E\max_{B\in\bC^\prime}G_B: \bC^\prime\subset \bC \mbox{ is finite }\right\}.
\end{equation}This supremum is called the lattice supremum (see Chapter~2.2 in \cite{MR2814399} for more details).

In the following, we consider a finite dimensional Hilbert space $H$ and we denote by $d$ its dimension. The two examples analyzed in Section~\ref{sec:examples} are $H=\R^p$ and $H=\R^{k\times m}$.
In this case, for all $\bC\subset H$ 
 we have $$\ell^*(\bC) = \E\sup_{B\in\bC}\inr{G, B}$$
 where $G$ is a  $H$-valued random variable with i.i.d. $\cN(0, 1)$ components. 
We will also need the following definition, cf. \cite{shahar_expect}.

\begin{definition}\label{def:unconditional}
Let $H$ be a finite dimensional Hilbert space, let $(e_j)_{j=1, \ldots, d}$ be a basis in $H$, and $K>0$. A norm $\norm{\cdot}$ on $H$ is called $K$-unconditional norm with respect to the basis $(e_j)_{j=1, \ldots, d}$ if the following two properties hold.
 \begin{itemize}
     \item
For any $B\in H$ and any permutation $\pi$ of $\{1, \ldots,d\}$,
\begin{equation*}
\Big\|{\sum_{j=1}^d \inr{B, e_j} e_j} \Big\|
\leq K \Big\|{\sum_{j=1}^d\inr{B, e_{\pi(j)}} e_j}\Big\|.
\end{equation*}
\item If $A\in H$ is such that $\inr{A, e_j}^\sharp\leq \inr{B, e_j}^\sharp$ for all $j=1, \ldots,d$, then
\begin{equation*}
\Big\|{\sum_{j=1}^d \inr{A, e_j} e_j} \Big\|
\leq K \Big\|{\sum_{j=1}^d\inr{B, e_j} e_j}\Big\|
\end{equation*}
where $(\inr{B, e_j}^\sharp)_{j}$ is the nonincreasing rearrangement of $(|\inr{B, e_j}|)_{j}$.
\end{itemize}
 \end{definition} 

 The class of $K$-unconditional norms is rather rich. It includes, in particular, the $\ell_p$-norms.  For more details see \cite{shahar_expect}.
 
 A bound on the expectation of the empirical process  in \eqref{eq:minimal_n_re} can be obtained from the following result.  

 \begin{theorem}{\rm \cite[Theorem~1.6]{shahar_expect}}
 \label{theo:shahar} There exists an absolute constant $c_1>0$ such that the following holds.
  Let $H$ be a finite dimensional Hilbert space, let $X$ be a random vector with values in $H$ and let $\bC\subset H$. Denote by $(e_j)_{j=1,\ldots, d}$ a basis in $H$.
 Let $L\ge 1$, and assume that:
 \begin{itemize}
     \item[(i)] The set $\bC$ is such that $\norm{\cdot}_{\bC^\circ}\triangleq \sup_{v\in\bC}\inr{v,\cdot}$ is a $K$-unconditional norm.
     \item[(ii)] The distribution of $X$ is isotropic,
         i.e., satisfies \eqref{isotropic}, and for any $j=1, \ldots, d$, and
         any positive integer $k$ smaller than $c_1 \log d$ we have 
         $$(\E \vert \inr{X, e_j} \vert^k)^{1/k}\leq L \sqrt{k}.$$
 \end{itemize}
      Let $X_1, \ldots, X_n$ be i.i.d. realizations of $X$ and let $\eps_1, \ldots, \eps_n$ be i.i.d. random variables uniformly distributed on $\{-1, 1\}$ and independent of $X_1, \ldots, X_n$. Then
     \begin{equation}
         \label{upper-bound-ell-star}
         \E \sup_{B\in\bC}\left|\frac{1}{\sqrt{n}}\sum_{i=1}^n \eps_i\inr{X_i, B}\right|\leq C(L,K) \ell^*(\bC),
     \end{equation}
     where $C(L,K)$ is a constant that depends only on $K$ and $L$.
 \end{theorem}

 If  condition (i) of this theorem does not hold, i.e., if 
 $\|\cdot\|_{\mathbb C^\circ}$ is not an unconditional norm,
 one may derive a similar result under a more constraining assumption, namely that the random variable $\inr{X,B}$ is subgaussian for any $B\in H$.
The next proposition follows from the majorizing measure theorem,
 cf.
 \cite[Chapter 1]{MR3184689} or
 \cite[Chapter 6]{handel_lecture_notes}. 
 
\begin{proposition}
    \label{prop:subgaussian}
    Let $L\ge 1$ and let $H$ be a finite dimensional Hilbert space.
    Assume that $X$ is isotropic, i.e., it satisfies \eqref{isotropic}.
    Assume also that $X$ is $L$-subgaussian in the sense that
    for all $B \in H$ such that $\|B\|_H = 1$ we have 
    $\E\exp(t\inr{X,B}) \le \exp(t^2L^2/2)$ for all $t>0$.
    Then $X$ satisfies the small ball assumption with parameters
     $\beta_0,\kappa_0$ defined in \eqref{sb-params}.
    Furthermore, there exists an absolute constant $c_2>0$
    such that \eqref{upper-bound-ell-star} holds with $C(L, K)=c_2 L$ for any $\mathbb C\subset H$.
\end{proposition}
\begin{proof}
    Let $Z=\inr{X,B}$ and assume w.l.o.g. that $\|B\|_H = 1$.
    The random variable $Z$ is $L$-subgaussian and, by isotropy,  $\E Z^2 = 1$.
    Thus
    by \cite[(2.3) from Theorem 2.1]{MR3185193} we have
    $\E Z^4 \le 16 L^4$, or equivalently $\E[Z^4]^{1/4} \le 2L \E[Z^2]^{1/2}$.
    By Lemma~\ref{lemma:moments2-4}, this implies that $X$ satisfies the small ball assumption with parameters
    $\beta_0,\kappa_0$ defined in \eqref{sb-params}.

    To prove \eqref{upper-bound-ell-star}, note that $\eps_i X_i$ is $L$-subgaussian. Thus, \eqref{upper-bound-ell-star}
    with $C(L, K)=c_2 L$ follows from the majorizing measure theorem for
    subgaussian processes, cf. \cite[Corollary 6.26]{handel_lecture_notes}.
\end{proof}


\section{Examples}\label{sec:examples}

In what follows, we denote by $|\cdot|_q$ the $\ell_q$ norm of a finite dimensional vector, $1\le q\le\infty$. We denote by $\|\cdot\|_{Fr}$ and by $\|\cdot\|_{sp}$
the Frobenius norm and the spectral norm of a matrix, respectively. Let $S_2^{p-1}$ and $B_q^{p}$ denote the unit Euclidean sphere in $\R^{p}$ and the unit $\ell_q$-ball in $\R^{p}$, respectively.
 The canonical basis of $\R^p$ is denoted by $(e_j)_{j=1, \ldots, p}$. For a vector  $\vbeta\in\R^p$ and  a  subset $S\subseteq \{1,\dots, p\}$, we denote by ${\rm supp}(\vbeta)$ the  support of $\vbeta$, by $\vbeta_S$ the orthogonal projection of $\vbeta$ onto the linear span
of $\{e_j:j\in S \}$, and by  $|S|$ the cardinality of $S$. We will write $a\lesssim b$ if there is an absolute constant $C>0$
such that $a\le Cb$.

\subsection{LASSO}

We consider here $H=\bB=\R^p$ equipped with the Euclidean norm $\|\cdot\|_H=|\cdot|_2$ and we define the regularization norm $\|\cdot\|$ as the $\ell_1$ norm. Then the estimator $\hbeta$ is the LASSO estimator
\begin{equation}\label{eq:def-lasso}
    \hbeta \in \argmin_{\vbeta\in \R^p}
    \left(
        \norm{\design\vbeta - \vy}_n^2
        + \lambda |\vbeta|_1
    \right)
\end{equation}
where $\lambda>0$ is a tuning parameter.

Given $\vbeta\in\R^p$ it is straightforward to verify that Assumption~\ref{assum:decomp} is satisfied when $\cP_\vbeta$ is the orthogonal projection operator onto the linear span
of $\{e_j:j\in {\rm supp}(\vbeta) \}$ where  $(e_j)_{j=1, \ldots, p}$ is the canonical basis of $\R^p$. 

 The operator $\design$ is a matrix in $\R^{n\times p}$ while the event $\Omega$
 in Theorem~\ref{th1} can be
written in the form
 \begin{equation}
        \Omega =\left\{  \,\frac 1 n |\design^T \vxi |_\infty \le \tau' \right\}.
        \label{eq:lasso-event}
        \end{equation}
        In order to apply Theorem~\ref{th1}, we need to find $\tau^\prime$ such that $\bP(\Omega)\geq 1/2$. Assume first
that $\design$ is deterministic.  The following lemma is a direct consequence of the normal tail probability bounds and the union bound, cf. \cite{bel-tsyb}.
    
    \begin{lemma} \label{lem:lass}
Let $(e_j)_{j=1, \ldots, p}$ be the canonical basis of $\R^p$ and let $\design$ be deterministic. If
\begin{equation}\label{lambda:lasso}
\tau' \ge \sigma \max_{1\leq j\leq p}\norm{\design e_j}_n \sqrt{\frac{2\log p}{n}},
\end{equation}  
then the event \eqref{eq:lasso-event} has probability at least 1/2.
\end{lemma}

In view of this lemma, oracle inequalities for the LASSO estimator with tuning parameter $\lambda$ satisfying
\begin{equation}\label{eq:lasso_param}
\lambda\ge 10 \sigma \max_{1\leq j\leq p}\norm{\design e_j}_n \sqrt{\frac{2 \log p}{n}} 
\end{equation} follow from Theorem~\ref{th1} and Corollary~\ref{cor1}. They have the following form.

\begin{theorem}\label{theo:lasso}
Assume that $\vxi\sim \mathcal N(\vzero,\sigma^2 I_{n\times n})$ and that $\design$ is deterministic.
Let  $\delta\in (0,1)$. The LASSO estimator $\hbeta$  with tuning parameter satisfying \eqref{eq:lasso_param} is such that, with probability at least $1-\delta$,
\begin{equation*}
\|\design\hbeta-\vf\|_n^2\le\min_{\vbeta \in \bR^p }\Big[\|\design\vbeta-\vf\|_n^2+\frac{16}{25}\lambda^2 \mu_4^2 (\vbeta)\Big] + \frac{16\sigma^2(\Phi^{-1}(1-\delta))^2}{n}
\end{equation*}and 
\begin{equation*}
\E \|\design\hbeta-\vf\|_n^2\le\min_{\vbeta \in \bR^p}\Big[\|\design\vbeta-\vf\|_n^2+\frac{16}{25}\lambda^2 \mu_4^2 (\vbeta)\Big] + \frac{16\sigma^2}{n}.
\end{equation*} 
 If, in addition, $\vf =\design\vbeta^* $ for some $\vbeta^* \in \bR^p$, then with probability at least $1-\delta$,
\begin{equation*}
|\hbeta-\vbeta^*|_1\le4\lambda \mu_4^2 (\vbeta^*)+\frac{20\sigma^2 (\Phi^{-1}(1-\delta))^2}{n \lambda}    
\end{equation*}
and
\begin{equation*}
\E |\hbeta-\vbeta^*|_1\le8\lambda \mu_4^2 (\vbeta^*)+ \frac{20\sigma }{\lambda n}.
\end{equation*}
\end{theorem}

To make these inequalities more explicit, we need to control the compatibility factor $\mu_{c_0}(\vbeta)$. First note that one may use the Restricted Eigenvalue constant \cite{bickel2009simultaneous} to bound $\mu_{c_0}(\vbeta)$ from above. For any $S\subset\{1,\dots,p\}$ and $c_0> 0$, we define the Restricted Eigenvalue constant $\kappa(S,c_0)\ge 0$ by the formula  
\begin{equation}
    \kappa^2(S,c_0) \triangleq \min_{\vdelta\in\R^p\setminus \{0\}: |\vdelta_{S^c}|_1\le c_0|\vdelta_{S}|_1} \frac{\norm{\design\vdelta}_n^2}{|\vdelta|_2^2}.
    \label{eq:def-kappa}
\end{equation}
Therefore, for all $\vbeta$ such that $ \kappa^2({\rm supp}(\vbeta),c_0)\ne 0$ we obtain the bound 
$$
\mu_{c_0}^2(\vbeta)\le \frac{ |{\rm supp}(\vbeta)|}{\kappa^2({\rm supp}(\vbeta),c_0)}.
$$
When $\design$ is deterministic and $\vbeta$ is $s$-sparse (i.e., $|{\rm supp}(\vbeta)|\le s$), there exist various sufficient conditions on $\design$ allowing one to bound $\kappa^2({\rm supp}(\vbeta),c_0)$ from below by a universal constant, cf., e.g., \cite{bickel2009simultaneous}. This leads to the bound $\mu_{c_0}^2(\vbeta) \lesssim s$ for all $s$-sparse vectors $\vbeta$.

Consider now the case of random $\design$. Specifically, assume that $X_1, \ldots, X_n$ are i.i.d. realizations of a random vector $X$ with values in $\R^p$.  Then, it turns out that the bound $\mu_{c_0}^2(\vbeta) \lesssim s$ for $s$-sparse vectors $\vbeta$ can be guaranteed with high probability (with respect to the distribution of $X_1, \ldots, X_n$) provided that $n\gtrsim s\log(ep/s)$. 
Indeed, combining
Theorems~\ref{theo:small_ball_re} and~\ref{theo:shahar} we obtain the following result.

\begin{proposition}\label{prop:constant_mu_lasso}
Let $L\ge 1$ and let $\beta_0,\kappa_0$ be positive constants. 
There exist a constant $C(L)>0$ depending only on~$L$ and an absolute constant $c_1>0$ such that the following holds. 

Let $X_1, \ldots, X_n$ be i.i.d. realizations of a random vector $X$ with values in $\R^p$ such that 
\begin{itemize}
    \item[(i)] $X$ satisfies the small ball assumption (Assumption~\ref{ass:small_ball}) with parameters $\beta_0,\kappa_0$,
    \item[(ii)] $X$ is isotropic (i.e., $\E X X^\top = I_{p\times p}$) and for all
        $j=1, \ldots, p,$ and all positive integers $k$ smaller than $c_1\log p$
        we have
         $(\E \vert \inr{X, e_j} \vert^k)^{1/k}\leq L \sqrt{k}$.
\end{itemize}
Let $s\in\{1, \ldots, p\}$ and $c_0>0$.  If 
\begin{equation}
    \label{condition-n-s-p-lasso}
    n\geq C(L)[(1+c_0)/(\kappa_0\beta_0)]^2 s \log(ep/s),
\end{equation}
then with probability greater than $1-\exp(-n\kappa_0^2/32)$, for every $\vbeta\in\R^p$ such that $|{\rm supp}(\vbeta)|\leq s$ we have
\begin{equation*}
    \mu_{c_0}(\vbeta)\leq \sqrt{\frac{8|{\rm supp}(\vbeta)|}{\beta_0^2 \kappa_0}}.
\end{equation*}
\end{proposition}

\begin{proof}
    Denote by $B_0(s)$ the set of all $s$-sparse vectors in $\R^p$:  $$B_0(s) = \{\vbeta\in \R^p: \,|{\rm supp}(\vbeta)|\leq s\}.$$ Let $\vbeta\in B_0(s)$ and recall that 
    $$\bC_{\vbeta,c_0}=\left\{\vbeta^\prime\in\R^p:|\cP_\vbeta^\perp\vbeta^\prime |_1\leq c_0 |\cP_{\vbeta}\vbeta^\prime|_1\right\}$$ where $\cP_\vbeta$ is the projection operator onto the linear span of $\{e_j:j\in {\rm supp}(\vbeta) \}$. It follows from Theorem~\ref{theo:small_ball_re} and \eqref{eq:condi_mu} that, if \eqref{eq:minimal_n_re}  with $\bA=B_0(s)$ holds, then with probability greater than $1-\exp(-n\kappa_0^2/32)$, for all $\vbeta\in B_0(s)$ we have 
\begin{equation*}
    \mu_{c_0}(\vbeta)\leq \left(\frac{8}{\beta_0^2\kappa_0}\right)^{1/2}\sup_{\vbeta^\prime\in\bC_{\vbeta,c_0}}\frac{|\cP_\vbeta \vbeta^\prime|_1}{|\vbeta^\prime|_2}\leq \left(\frac{8}{\beta_0^2\kappa_0}\right)^{1/2}\sqrt{|{\rm supp}(\vbeta)|}
\end{equation*}
where we have used that $|\cP_\vbeta\vbeta^\prime|_1\leq \sqrt{|{\rm supp}(\vbeta)|}|\vbeta^\prime|_2$ for all $\vbeta^\prime\in\R^p$. 

Therefore, it only remains to prove that \eqref{condition-n-s-p-lasso} implies \eqref{eq:minimal_n_re} with $\bA=B_0(s)$. First note that $S_2^{p-1}\cap \left(\cup_{\vbeta\in B_0(s)}\bC_{\vbeta, c_0}\right)\subset \bC$ where $\bC= \left((1+c_0)\sqrt{s}B_1^p\right)\cap B_2^{p}$. Since the $\ell_2$ and $\ell_1$ norms are $1$-unconditional, it is straightforward to check that $\norm{\cdot}_{\bC^\circ}=\sup_{v\in\bC}\inr{v,\cdot}$ is a $1$-unconditional norm. Therefore, we can apply Theorem~\ref{theo:shahar}, which gives 
\begin{align*}
   \E \sup_{\vbeta\in S_2^{p-1}\cap \left(\cup_{\vbeta\in B_0(s)}\bC_{\vbeta, c_0}\right)}\left|\frac{1}{n}\sum_{i=1}^n\eps_i \inr{X_i, \vbeta}\right|
   &\leq
   \E \sup_{\vbeta\in \bC}\left|\frac{1}{n}\sum_{i=1}^n\eps_i \inr{X_i, \vbeta}\right|, \\
   &\leq 
   \frac{c_2(L)\ell^*(\bC)}{\sqrt{n}}
   \leq
   \frac{c_3(L)(1+c_0)\sqrt{s \log(ep/s)}}{\sqrt{n}},
\end{align*}
where $c_2(L)$ and $c_3(L)$ are positive constants depending only on $L$ and where
we used that $\ell^*(\bC) \leq (1+c_0)\ell^*(\sqrt{s}B_1^p\cap B_2^{p})\leq c_4 (1+c_0) \sqrt{s \log(ep/s)}$ for some absolute constant $c_4$ (cf., for instance, Lemm{}a~5.3 in \cite{reg1}).
If \eqref{condition-n-s-p-lasso} holds with large enough constant $C(L)>0$ depending only on $L$,  then
the right hand side of the last display is bounded from above by $\beta_0\kappa_0/16$
and \eqref{eq:minimal_n_re} is satisfied.
\end{proof}

Combining Theorem~\ref{theo:lasso} and Proposition~\ref{prop:constant_mu_lasso} we can obtain oracle inequalities for the LASSO estimator when $X_1,\dots,X_n$ are i.i.d. random vectors {independent of the noise vector $\vxi$}. To illustrate it, consider the following result for the basic example where all entries of matrix $\design$ are i.i.d. standard Gaussian.

\begin{theorem}\label{theo:lasso:random}
Assume that $\vxi\sim \mathcal N(\vzero,\sigma^2 I_{n\times n})$ and that all entries of matrix $\design$ are i.i.d. standard Gaussian
random variables {independent of the noise vector $\vxi$}. 
Let  $\delta\in (0,1)$, and 
\begin{equation}\label{eq:lasso_param1}
\lambda = a \sigma  \sqrt{\frac{2 \log p}{n}} 
\end{equation}
where $a\ge 20$. There exist an absolute constant $C_1>0$ and a constant $C_2>0$ depending only on $a$ such that the following holds.
If $ n\geq C_1 s \log(ep/s)$, then for
the LASSO estimator $\hbeta$  with tuning parameter  \eqref{eq:lasso_param1} we have that, with probability at least $1-\delta-(p+1)e^{-n/C_1}$,
\begin{equation*}
\|\design\hbeta-\vf\|_n^2\le\min_{\vbeta \in B_0(s) }\Big[\|\design\vbeta-\vf\|_n^2+C_2 \sigma^2 \frac{|{\rm supp}(\vbeta)| \log p}{n}\Big] + \frac{16\sigma^2(\Phi^{-1}(1-\delta))^2}{n}.
\end{equation*}
If, in addition, $\vf =\design\vbeta^* $ for some $\vbeta^* \in B_0(s)$, then with probability at least $1-\delta-(p+1)e^{-n/C_1}$,
\begin{equation*}
|\hbeta-\vbeta^*|_1\le C_2 \sigma  \Big( s \sqrt{\frac{ \log p}{n}}+  \frac{(\Phi^{-1}(1-\delta))^2}{\sqrt{n\log p}}    \Big).
\end{equation*}
\end{theorem}
\begin{proof}  We first plug the bound on $\mu_4$ given by Proposition~\ref{prop:constant_mu_lasso} into the oracle inequalities in deviation of Theorem~\ref{theo:lasso}. Then, we obtain resulting oracle inequalities that hold with probability $1-\delta-e^{-n/C_1}$ for all $\vbeta \in B_0(s)$. To finish the proof, it suffices to compare the definitions of $\lambda$ in \eqref{eq:lasso_param} and in \eqref{eq:lasso_param1}, and notice that 
\begin{equation}\label{eq:soilasso-21}
\bP(\max_{1\leq j\leq p}\norm{\design e_j}_n \le 2)\ge 1-pe^{-n/2}. 
\end{equation}
Indeed, the random variable $\zeta_j=\norm{\design e_j}_n$ is a $1/\sqrt{n}$-Lipschitz function of the standard Gaussian vector in $\R^n$. Thus, by the Gaussian concentration inequality, cf., e.g., \cite[Theorem 5.6]{MR3185193}, we get $\bP(\zeta_j>2)\le \bP(\zeta_j-\bE(\zeta_j)>1)\le e^{-n/2}$, where we have used that $\bE(\zeta_j)\le (\bE (\zeta_j^2) )^{1/2} =1$. This and the union bound yield \eqref{eq:soilasso-21}.
\end{proof}

\subsection{Group LASSO}

We consider here $H=\bB=\R^p$ equipped with the Euclidean norm $\|\cdot\|_H=|\cdot|_2$ and define the regularization norm $\norm{\cdot}$ as follows. Let $G_1,\ldots,G_M$ be a partition of $\{1,\ldots,p\}$. For any $\vbeta\in\R^p$ we set
\begin{equation}\label{glasso_norm}
\norm{\vbeta} = |\vbeta|_{2,1} \triangleq \sum_{k=1}^M |\vbeta_{G_k}|_2.
\end{equation}
The group LASSO estimator is a solution of the convex minimization problem
\begin{equation}
    \hbeta \in \argmin_{\vbeta\in\R^p}\left(
    \norm{\vy-\design\vbeta}_n^2
    + \lambda \sum_{k=1}^M |\vbeta_{G_k}|_2\right),
    \label{eq:def-glasso}
\end{equation}
 where $\lambda>0$ is a tuning parameter. In the following, we assume that the groups $G_k$ have the same cardinality $|G_k|= T=p/M$, $k=1,\dots,M$. 
 
 To any $\vbeta\in \R^p$, we associate the set $${\mathcal K}(\vbeta)=\{k\in\{1,...,M\}: \vbeta_{G_k}\ne \vzero\},$$ 
 which plays the role of ``support by block'' of vector $\vbeta$. Given $\vbeta\in\R^p$, it is straightforward to check that Assumption~\ref{assum:decomp} is satisfied when $\cP_\vbeta$ is the orthogonal projection operator onto the linear span of $\left\{e_j:j \in \cup_{k\in\cK(\vbeta)} G_{k}\right\}$.

The operator $\design$ is a matrix in $\R^{n\times p}$ while the event $\Omega$
 in Theorem~\ref{th1}  takes now the form
 \begin{equation}
        \Omega =\left\{ \max_{k=1,...,M} \,\frac 1 n |\design_{G_k}^T \vxi |_2 \le \tau' \right\}
        \label{eq:glasso-event}
    \end{equation}
    where  $\design_{G_k}$ is the $n\times |G_k|$ submatrix of $\design$ composed from all the columns of $\design$ with indices in $G_k$.

In order to apply Theorem~\ref{th1}, we need to find $\tau^\prime$ such that $\bP(\Omega)\geq 1/2$.     
Denote by $\|\design_{G_k}\|_{sp}\triangleq \sup_{|\vx|_2\le 1}|\design_{G_k}\vx|_2$ the spectral norm of matrix $\design_{G_k}$ and by $\|\design_{G_k}\|_{Fr}$ its Frobenius norm.  Then, set $\psi_{sp}^*  =\max_{k=1,...,M}\|\design_{G_k}\|_{sp}/\sqrt{n}$ and $\psi_{Fr}^*  =\max_{k=1,...,M}\|\design_{G_k}\|_{Fr}/\sqrt{n}$. The constant $\tau'$ is determined by the following straightforward modification of Lemma~2 in \cite{bel-tsyb}.

\begin{lemma} \label{lem:group}
Let $\design$ be deterministic. If 
\begin{equation}\label{lambda:glasso}
\tau' \ge \frac{\sigma}{\sqrt{n}}\left(\psi^*_{Fr} + \psi_{sp}^* \sqrt{2\log(2 M)}\right),
\end{equation}  
then the event \eqref{eq:glasso-event} has probability at least 1/2.
\end{lemma}

Using this lemma, oracle inequalities for the group LASSO estimator with tuning parameter $\lambda$ satisfying 
\begin{equation}\label{eq:group_lasso_param}
\lambda\ge \frac{10\sigma}{\sqrt{n}}\left(\psi^*_{Fr} + \psi_{sp}^* \sqrt{2\log(2 M)}\right)  
\end{equation} 
can be deduced from Theorem~\ref{th1} and Corollary~\ref{cor1}. They have the following form.

\begin{theorem}\label{theo:glasso}
Assume that $\vxi\sim \mathcal N(\vzero,\sigma^2 I_{n\times n})$ and that $\design$ is deterministic.
Let  $\delta\in (0,1)$. The group LASSO estimator $\hbeta$  with tuning parameter satisfying \eqref{eq:group_lasso_param} is such that, with probability at least $1-\delta$,
\begin{equation*}
\|\design\hbeta-\vf\|_n^2\le\min_{\vbeta \in \bR^p }\Big[\|\design\vbeta-\vf\|_n^2+\frac{16}{25}\lambda^2 \mu_4^2 (\vbeta)\Big] + \frac{16\sigma^2(\Phi^{-1}(1-\delta))^2}{n}
\end{equation*}and 
\begin{equation*}
\E \|\design\hbeta-\vf\|_n^2\le\min_{\vbeta \in \bR^p}\Big[\|\design\vbeta-\vf\|_n^2+\frac{16}{25}\lambda^2 \mu_4^2 (\vbeta)\Big] + \frac{16\sigma^2}{n}.
\end{equation*} 
 If, in addition, $\vf =\design\vbeta^* $ for some $\vbeta^* \in \bR^p$, then with probability at least $1-\delta$,
\begin{equation*}\label{eq:soiglasso-2}
|\hbeta-\vbeta^*|_{2,1}\le4\lambda \mu_4^2 (\vbeta^*)+\frac{20\sigma^2 (\Phi^{-1}(1-\delta))^2}{n \lambda}    
\end{equation*}
and
\begin{equation*}
\E |\hbeta-\vbeta^*|_{2,1}\le8\lambda \mu_4^2 (\vbeta^*)+ \frac{20\sigma }{\lambda n}.
\end{equation*}
\end{theorem}

Consider now a control of parameter $\mu_{c_0}(\vbeta)$ for vectors $\vbeta$ with a ``sparse by block" structure.  To that end, one can use the ``group" analog of the Restricted Eigenvalue constant, cf.~\cite{lptv2011}. For any $S\subset\{1,...,M\}$ and $c_0> 0$,
we define the group Restricted Eigenvalue constant $\kappa_G(S,c_0)\ge 0$ by the formula
\begin{equation}
    \kappa^2_G(S,c_0) \triangleq \min_{\vdelta\in\mathcal C(S,c_0)} \frac{\norm{\design\vdelta}_n^2}{|\vdelta|_2^2},
    \label{eq:def-kappa-g}
\end{equation}
where $\mathcal C(S,c_0)$ is the cone
$$
\mathcal C (S,c_0)\triangleq \left\{\vdelta\in\R^p\setminus \{0\}: \sum_{k\in S^c}|\vdelta_{G_k}|_2\le c_0\sum_{k\in S}|\vdelta_{G_k}|_2\right\}.
$$
In particular, for all $\vbeta\in\R^p$ with $\kappa_G({\mathcal K}(\vbeta),c_0)\ne 0$ we have
\begin{equation*}
\mu_{c_0}^2(\vbeta)\leq \frac{|\cK(\vbeta)|}{\kappa_G^2(\cK(\vbeta), c_0)}.
\end{equation*}
When $\design$ is deterministic and $\vbeta$ is such that $|\cK(\vbeta)|\leq s$
 sufficient conditions on $\design$ allowing one to bound $\kappa_G^2({\rm supp}(\vbeta),c_0)$ from below by a universal constant can be found in~\cite{lptv2011}. This leads to the bound $\mu_{c_0}^2(\vbeta) \lesssim s$ for all vectors $\vbeta$ such that $|\cK(\vbeta)|\leq s$ (i.e., all {\it $s$-sparse by block vectors}).

Finally, we give an upper bound on $\mu_{c_0}(\vbeta)$ in the case of random $\design$. Let $X_1, \ldots, X_n$ be i.i.d. realizations of a random vector $X$ with values in $\R^p$.  The following proposition shows that, with high probability (with respect to the distribution of  $X_1, \ldots, X_n$), we have $\mu_{c_0}^2(\vbeta)\lesssim |\cK(\vbeta)|$ for all $s$-sparse by block vectors $\vbeta\in\R^p$ provided that $n\gtrsim s(T +  \log(M/s))$.

\begin{proposition} Let $L\ge 1$. Let $X_1, \ldots, X_n$ be i.i.d. realizations of a random vector $X$ with values in $\R^p$ such that 
\begin{itemize}
    \item[(i)] $X$ is isotropic (i.e., $\E X X^\top = I_{p\times p}$),    
    \item[(ii)] $X$ is $L$-subgaussian: $\bE\exp\left(t\inr{X, \vbeta}\right)\leq \exp(L^2 t^2/2)$ for all $t>0$ and all $ \vbeta\in\R^{p}$ such that $| \vbeta|_2=1$.
    \end{itemize}
Let $s\in\{1, \ldots, M\}$ and $c_0>0$. There exist positive constants $C(L)$ and $C'(L)$ depending only on~$L$ such that the following holds. If 
\begin{equation}
    \label{condition-s-T-M-group-lasso}
     n\geq C(L) (1+c_0)^2 s \left(T+\log(M/s)\right)  
\end{equation}
then with probability greater than $1-\exp(-C'(L) n)$, for any $ \vbeta\in\R^{p}$ such that $|\cK(\vbeta)|\leq s$ we have 
\begin{equation*}
    \mu_{c_0}( \vbeta)\leq 32 L^2\sqrt{|\cK(\vbeta)|}.
\end{equation*}
\end{proposition}

\begin{proof}
 Since $X$ is $L$-subgaussian and isotropic, it follows from Proposition~\ref{prop:subgaussian}
 that $X$ satisfies the small ball assumption with parameters
 $\beta_0,\kappa_0$ defined in \eqref{sb-params}.

The definition of $\norm{\cdot}$ in \eqref{glasso_norm} and the fact that $\cP_\vbeta$ is the projection operator onto the linear span of $\left\{e_j:j\in\cup_{k\in\cK(\vbeta)}G_k\right\}$ imply
\begin{eqnarray*}\bC_{ \vbeta,c_0}
&=&
\Big\{ \vbeta^\prime\in\R^{p}:\norm{\cP_ \vbeta^\perp \vbeta^\prime }\leq c_0 \norm{\cP_{ \vbeta} \vbeta^\prime}\Big\} 
\\ &=& 
\Big\{\vbeta^\prime\in\R^p: \sum_{k\in \cK(\vbeta)^c}|\vbeta_{G_k}^\prime|_2\le c_0\sum_{k\in \cK(\vbeta)}|\vbeta_{G_k}^\prime|_2\Big\}.
\end{eqnarray*}
Denote by $\Sigma_s$ the set of all vectors $\vbeta$ in $\R^{p}$ such that $|\cK(\vbeta)|\leq s$. 
It follows from Theorem~\ref{theo:small_ball_re} and \eqref{eq:condi_mu} that, if \eqref{eq:minimal_n_re} holds with $\bA=\Sigma_s$,  then with probability at least $1-\exp(-n\kappa_0^2/32)$, for all $\vbeta\in \Sigma_s$  we have
\begin{equation*}
    \mu_{c_0}( \vbeta)
    \leq
    \left(\frac{8}{\beta_0^2\kappa_0}\right)^{1/2}\sup_{ \vbeta^\prime\in\bC_{ \vbeta,c_0}}\frac{\norm{\cP_\vbeta  \vbeta^\prime}}{| \vbeta^\prime|_2}
    \leq
    \sqrt{\frac{8|\cK( \vbeta)|}{\beta_0^2\kappa_0} } = 32 L^2 \sqrt{|\cK( \vbeta)|} 
\end{equation*}
since $\norm{\cP_\vbeta \vbeta^\prime}\leq \sqrt{|\cK(\vbeta)|}| \vbeta^\prime|_2$ for all $ \vbeta^\prime\in\R^{p}$. 

Therefore, it only remains to prove that  \eqref{condition-s-T-M-group-lasso} implies \eqref{eq:minimal_n_re} with $\bA=\Sigma_s$.
Denote by $B$ the unit ball with respect to the group LASSO norm  $\norm{\cdot}$ in $\R^{p}$. It is straightforward to check that $S_2^{p-1}\cap \left(\cup_{ \vbeta\in\Sigma_s}\bC_{ \vbeta, c_0}\right)\subset \bC$ where $\bC= \left((1+c_0)\sqrt{s}B \right)\cap B_2^{p}$. 
By Proposition~\ref{prop:subgaussian}, we have,
for an absolute constant $c_2>0$,
\begin{equation*}
   \E \sup_{ \vbeta\in S_2^{p-1}\cap \left(\cup_{ \vbeta\in\Sigma_s}\bC_{ \vbeta, c_0}\right)}\left|\frac{1}{n}\sum_{i=1}^n\eps_i \inr{X_i,  \vbeta}\right|\leq  \E \sup_{ \vbeta\in \bC}\left|\frac{1}{n}\sum_{i=1}^n\eps_i \inr{X_i,  \vbeta}\right|\leq \frac{c_2 L \ell^*(\bC)}{\sqrt{n}}.
\end{equation*}
We now bound $\ell^*(\bC)$ from above.
First note that $\ell^*(\bC) \leq (1+c_0)\ell^*(\sqrt{s}B\cap B_2^{p})$. Denote by $\veta=(\eta_j)_{j=1}^p$ a standard Gaussian vector in $\R^p$ and by $(\zeta^\sharp_k)_{k=1}^M$ a non-increasing rearrangement of $(|\veta_{G_k}|_2)_{k=1}^M$.  We have
\begin{align*}
&\ell^*(\sqrt{s}B\cap B_2^{p}) = \E \sup\left(\sum_{k=1}^M t_k \zeta_k:\sum_{k=1}^M |t_k|\leq \sqrt{s}, \sum_{k=1}^M t_k^2\leq1\right)\\
&\leq \E \sup\left(\sum_{k=1}^M t_k^\sharp \zeta_k^\sharp:\sum_{k=1}^M |t_k|\leq \sqrt{s}, \sum_{k=1}^M t_k^2\leq1\right)\\
&\leq \E \sup\left(\sum_{k=1}^s t_k^\sharp \zeta_k^\sharp: \sum_{k=1}^M t_k^2\leq1\right) + \E \sup\left(\sum_{k=s+1}^M t_k^\sharp \zeta_k^\sharp:\sum_{k=1}^M |t_k|\leq \sqrt{s}\right)\\
&=\E \sqrt{\sum_{k=1}^s (\zeta_k^\sharp)^2}
+ \sqrt{s}\E\max_{k=s+1, \ldots, M}\zeta_k^\sharp\leq 2 \E \sqrt{\sum_{k=1}^s (\zeta_k^\sharp)^2}
\\
&\leq \frac{ 2\sqrt{8s}}{\sqrt{3}}\left[\E \left(\frac{1}{s}\sum_{k=1}^s \frac{3(\zeta^\sharp_k)^2}{8}\right)\right]^{1/2}.
\end{align*}Then,  using Jensen's inequality we obtain
\begin{align*}
&\E \left(\frac{1}{s}\sum_{k=1}^s \frac{3(\zeta^\sharp_k)^2}{8}\right)
\leq 
\log \E \exp\left(\frac{1}{s}\sum_{k=1}^s \frac{3(\zeta^\sharp_k)^2}{8}\right)
\\
&\leq 
\log \E \exp\left(\frac{1}{s}\sum_{k=1}^M \frac{3|\veta_{G_k}|_2^2}{8}\right)
\le 
\log \left( \frac{1}{s}\sum_{k=1}^M\E\exp\left( \frac{3|\veta_{G_k}|_2^2}{8}\right)\right)\\
&= \log \left( \frac{1}{s}\sum_{k=1}^M \prod_{j\in G_k}\E\exp\left( \frac{3\eta_j^2}{8}\right)\right) = \log\left(\frac{2^TM}{s}\right).
\end{align*}
Therefore, there exists an absolute constant $c'>0$ such that 
\begin{equation*}
\frac{c_2L\ell^*(\bC)}{\sqrt{n}}
   \leq \frac{c'L(1+c_0)\sqrt{sT + s\log(M/s)}}{\sqrt{n}} .
\end{equation*}
For $\beta_0$ and $\kappa_0$ given in \eqref{sb-params}, the expression in the last display can be rendered smaller than $\kappa_0\beta_0 /16$ provided that \eqref{condition-s-T-M-group-lasso} holds with large enough constant $C(L)>0$ depending only on $L$. Thus, \eqref{eq:minimal_n_re} follows.
\end{proof}


\subsection{Nuclear norm penalty}

We consider here $H=\bB=\R^{k\times m}$ equipped with the Frobenius  norm $\|\cdot\|_H=\|\cdot\|_{Fr}$ and we define the regularization norm $\|\cdot\|$ as the nuclear norm $\|\cdot\|_*$ (i.e., the sum of the singular values). The corresponding penalized LS estimator $\hat A$ is a solution of the minimization problem
\begin{equation}\label{eq:def-nuclear}
    \hat A \in \argmin_{A\in \R^{k\times m}}
    \left(
        \norm{\design A - \vy}_n^2
        + \lambda \|A\|_*
    \right)
\end{equation}
where $\lambda>0$ is a tuning parameter. Penalized LS estimators with nuclear norm penalty were considered in several papers starting from \cite{rohde_tsyb,candes_plan,koltchinskii2011nuclear}. For more references, see \cite{koltch_book,giraud2014introduction,vdgeer}.

For
$A\in\R^{k\times m}$,  let $r={\rm
rank} (A)$ denote its rank. By the singular value
decomposition, $A=\sum_{j=1}^r \sigma_j(A)
u_j v_j^\top$ with orthonormal vectors $u_1,\dots, u_r\in {\mathbb
R}^{k}$, orthonormal vectors $v_1,\ldots,v_r\in {\mathbb R}^{m}$
and singular values $\sigma_1(A) \ge\dots\ge \sigma_r(A)>0$. The pair of linear vector spaces
$(S_1,S_2)$ where $S_1$ is the linear span of $\{u_1,\dots, u_r\}$
and $S_2$ is the linear span of $\{v_1,\dots, v_r\}$ will be called
the support of $A.$ We will denote by $S_j^{\perp}$ the
orthogonal complement of $S_j$, $j=1,2$, and by $P_{S}$ the
orthogonal projector on the linear vector space $S$. 
Given $A\in \bR^{k\times m}$ with support $(S_1,S_2),$ and $B\in {\mathbb
R}^{k\times m}$, we set
\begin{equation}\label{eq:proj_op_mat}
{\cal P}_A (B) \triangleq B-P_{S_1^{\perp}}BP_{S_2^{\perp}} \mbox{ and } {\cal
P}_A^{\perp}(B)\triangleq P_{S_1^{\perp}}BP_{S_2^{\perp}}.
\end{equation}

For the norm $\|\cdot\|=\|\cdot\|_*$,  Assumption \ref{assum:decomp} is satisfied with the operator $\cP_A$ defined in \eqref{eq:proj_op_mat}. 
Indeed, it is clear that $\mathcal P_A(A) = A$.
Furthermore, by definition of $\mathcal P_A^\perp$, the columns of $A$ are orthogonal { to the columns of} $\mathcal P_A^\perp(B)$
and the rows of $A$ are orthogonal to the rows of $\mathcal P_A^\perp(B)$.
Thus
$$\|A\|_* + \|\mathcal P_A^\perp(B) \|_* = \|A + \mathcal P_A^\perp(B) \|_*,$$
which means that  the nuclear norm satisfies Assumption  \ref{assum:decomp2} (the decomposability assumption), and a fortiori Assumption  \ref{assum:decomp}.

Oracle inequalities for the estimator \eqref{eq:def-nuclear} follow from Theorem~\ref{th1} and Corollary~\ref{cor1}. In order to apply those results,  one has to find $\tau^\prime$ such that $\bP(\Omega)\geq1/2$ where
\begin{equation}
    \label{def:Omega-matrices}
 \Omega=\left\{\sup_{\norm{B}_*\leq1}\frac{1}{n}\sum_{i=1}^n \xi_i \inr{X_i, B}\leq\tau^\prime\right\} = \left\{\frac{\norm{\Gamma}_{sp}}{\sqrt{n}}\leq \tau^\prime\right\},
 \end{equation}  
 $\Gamma=n^{-1/2}\sum_{i=1}^n \xi_i X_i$, $\norm{\Gamma}_{sp}$ is the spectral norm (i.e., the largest singular value of $\Gamma$), and $\xi_i$ are i.i.d. random variables with distribution $\cN(0,\sigma^2)$.
 The next result from \cite{rohde_tsyb} provides a control of the spectral norm of $\Gamma$. Define
 $$
 \phi_{max} \triangleq
     \sup_{\substack{ A\in\R^{k\times m}: \|A\|_{Fr}=1 \\ \text{ and rank}(A)= 1}}
     \left(
         \frac 1 n \sum_{i=1}^n \langle X_i, A\rangle^2 
     \right)^{1/2}.
 $$
 The quantity $\phi_{max}$ is the maximal rank-1 restricted eigenvalue of the operator $\design$.

\begin{lemma}[Lemma 2 in \cite{rohde_tsyb} with $D=2$]  \label{lem:control_op_norm_gamma}
    Let $k\ge2$ and $m\ge2$.
    Let $X_1,\dots,X_n$ be deterministic matrices in $\R^{k\times m}$
    and let $\xi_1,\dots,\xi_n$ be i.i.d. random variables with distribution $\cN(0,\sigma^2)$.
    If 
 \begin{equation*}
     \tau^\prime\geq 8 \sigma \phi_{max} \sqrt{\frac{k+m}{n}}
       \end{equation*} 
       then for the event $\Omega$ in \eqref{def:Omega-matrices}
  we have $\bP(\Omega)\geq 1-2\exp(-(2-\log 5)(m+k)) \ge 1/2$.
\end{lemma}


In view of this lemma, oracle inequalities for the nuclear norm regularization procedure \eqref{eq:def-nuclear} with tuning parameter  $ \lambda$ satisfying
\begin{equation}\label{eq:trace_param}
    \lambda\ge  80 \sigma \phi_{max}  \sqrt{\frac{k+m}{n}} 
\end{equation} 
can be deduced from Theorem~\ref{th1} and Corollary~\ref{cor1}. We have the following result.

\begin{theorem}\label{theo:trace}
    {Let $k\ge2$ and $m\ge2$.}
Assume that $\vxi\sim \mathcal N(\vzero,\sigma^2 I_{n\times n})$ and that $X_1,\dots ,X_n$ are deterministic matrices.
Let  $\delta\in (0,1)$. The estimator $ \hat A$ defined in \eqref{eq:def-nuclear} with tuning parameter satisfying \eqref{eq:trace_param} is such that, with probability at least $1-\delta$,
\begin{equation*}
\|\design \hat A-\vf\|_n^2\le\min_{A \in \bR^{k\times m} }\Big[\|\design A-\vf\|_n^2+\frac{16}{25}\lambda^2 \mu_4^2 ( A)\Big] + \frac{16\sigma^2(\Phi^{-1}(1-\delta))^2}{n}
\end{equation*}and 
\begin{equation*}
\E \|\design \hat A-\vf\|_n^2\le\min_{ A \in \bR^{k\times m}}\Big[\|\design A-\vf\|_n^2+\frac{16}{25}\lambda^2 \mu_4^2 ( A)\Big] + \frac{16\sigma^2}{n}.
\end{equation*} 
 If, in addition, $\vf =\design A^* $ for some $ A^* \in \bR^{k\times m}$, then with probability at least $1-\delta$,
\begin{equation*}
\| \hat A- A^*\|_*\le4\lambda \mu_4^2 ( A^*)+\frac{20\sigma^2 (\Phi^{-1}(1-\delta))^2}{n \lambda}    
\end{equation*}
and
\begin{equation*}
\E \| \hat A- A^*\|_*\le8\lambda \mu_4^2 ( A^*)+ \frac{20\sigma }{\lambda n}.
\end{equation*}
\end{theorem}

Finally, we give a bound on the compatibility factor $\mu_{c_0}( A)$ for low rank matrices $ A$
in the case where $X_1,\dots,X_n$ are i.i.d. random matrices. Using Theorem~\ref{theo:small_ball_re} we obtain 
the following result.

\begin{proposition}\label{prop:constant_mu_nuclear}
Let $L\ge 1$. Let $X_1, \ldots, X_n$ be i.i.d. realizations of a random matrix $X$ with values in $\R^{k\times m}$ such that 
\begin{itemize}
    \item[(i)] $X$ is isotropic: $\E \inr{X,  A}^2 = \| A\|_{Fr}^2$ for all $ A\in\R^{k\times m}$,    
    \item[(ii)] $X$ is $L$-subgaussian: $\bE\exp\left(t\inr{X, A}\right)\leq \exp(L^2 t^2/2)$ for all $t>0 $ and all $ A\in\R^{k\times m}$ such that $\| A\|_{Fr}=1$.
    \end{itemize}
    Let $s\in\{1, \ldots, \min(k,m)\}$ and $c_0>0$. There exist positive constants $c(L)$ and $c'(L)$ depending only on~$L$ such that the following holds. If 
\begin{equation}
    \label{condition-n-s-k-m}
    n\geq c(L)(1+c_0)^2 s \max(k,m),
\end{equation} 
then with probability greater than $1-\exp(-c'(L)n)$, for any $ A\in\R^{k\times m}$ such that ${\rm rank}( A)\leq s$ we have
\begin{equation*}
    \mu_{c_0}( A)\leq 32 L^2\sqrt{{\rm rank}(A)}.
\end{equation*}
\end{proposition}

\begin{proof}
    Since $X$ is $L$-subgaussian and isotropic, it follows from Proposition~\ref{prop:subgaussian}
    that $X$ satisfies the small ball assumption with parameters
    $\beta_0,\kappa_0$ defined in \eqref{sb-params}.

Denote by $M_s$ the set of all matrices in $\R^{k\times m}$ with rank at most $s$.
For any $ A\in\R^{k\times m}$ we have
$$\bC_{ A,c_0}=\left\{ A^\prime\in\R^{k\times m}:\norm{\cP_ A^\perp A^\prime }_*\leq c_0 \norm{\cP_{ A} A^\prime}_*\right\}
$$
where $\cP_A$ is the operator defined in \eqref{eq:proj_op_mat}. It follows from Theorem~\ref{theo:small_ball_re} and \eqref{eq:condi_mu} that, if \eqref{eq:minimal_n_re} holds with $\bA=M_s$,  then with probability at least $1-\exp(-n\kappa_0^2/32)$, for all $A\in M_s$ we have
\begin{equation*}
    \mu_{c_0}( A)
    \leq
    \left(\frac{8}{\beta_0^2\kappa_0}\right)^{1/2}\sup_{ A^\prime\in\bC_{ A,c_0}}\frac{\norm{\cP_A  A^\prime}_*}{\| A^\prime\|_{Fr}}
    \leq 
    \sqrt{\frac{8{\rm rank}( A)}{\beta_0^2\kappa_0} } = 32 L^2 \sqrt{{\rm rank}( A)} 
\end{equation*}
since $\norm{\cP_A A^\prime}_*\leq \sqrt{{\rm rank}( A)}\| A^\prime\|_{Fr}$ for all $ A^\prime\in\R^{k\times m}$. 

Therefore, it only remains to prove that \eqref{condition-n-s-k-m} implies \eqref{eq:minimal_n_re}  with $\bA=M_s$. Denote by $S_2^{km-1}$ and  $B_2^{k m}$ the unit Euclidean sphere and the unit Euclidean ball in $\R^{k\times m}$, respectively, and by $B_*$ the unit ball in $\R^{k\times m}$ with respect to the nuclear norm. It is straightforward to check that $S_2^{km-1}\cap \left(\cup_{ A\in M_s}\bC_{ A, c_0}\right)\subset \bC$ where $\bC= \left((1+c_0)\sqrt{s}B_* \right)\cap B_2^{km}$. 
Proposition~\ref{prop:subgaussian} yields that
\begin{align*}
   \E \sup_{ A\in S_2^{km-1}\cap \left(\cup_{ A\in M_s}\bC_{ A, c_0}\right)}\left|\frac{1}{\sqrt n}\sum_{i=1}^n\eps_i \inr{X_i,  A}\right|
    &\leq
    c_2 L\ell^*\left(
        S_2^{km-1}\cap \left(\cup_{ A\in M_s}\bC_{ A, c_0}\right)
    \right)\\
    &\le
    c_2 L\ell^*(\mathbb C).
\end{align*}
Next, by inclusion we have 
$\ell^*(\bC) \leq (1+c_0)\ell^*(\sqrt{s}B_*\cap B_2^{km})\leq (1+c_0)\sqrt{s}\ell^*(B_*)$.
By duality,  $\ell^*(B_*) = \E\|G\|_{sp}$ where $G$ is a random matrix with i.i.d. $\cN(0,1)$ entries. In addition, 
$\E\|G\|_{sp}\le \sqrt k + \sqrt m$, cf. \cite{davidson_szarek}.
Thus, for large enough constant $c(L)>0$, condition
\eqref{condition-n-s-k-m} implies \eqref{eq:minimal_n_re} with $\bA=M_s$.
\end{proof}

Combining Theorem~\ref{theo:trace} and Proposition~\ref{prop:constant_mu_nuclear} we can obtain oracle inequalities for the estimator $ \hat A$ defined in \eqref{eq:def-nuclear} when $X_1,\dots,X_n$ are i.i.d. random matrices {independent of the noise vector $\vxi$}. We illustrate it by the following result for the basic example where the entries of each of the matrices $X_i$ are i.i.d. standard Gaussian.

\begin{theorem}\label{theo:trace:random}
Assume that $\vxi\sim \mathcal N(\vzero,\sigma^2 I_{n\times n})$ and that $X_1,\dots,X_n$ are i.i.d. realizations of a random matrix $X$ whose entries are i.i.d. standard Gaussian
random variables. {We also assume that $X_1,\dots,X_n$ are independent of the noise vector $\vxi$}. 
Let  $\delta\in (0,1)$, $k\ge 2$, $m\ge 2$, and 
\begin{equation}\label{eq:trace_param1}
\lambda =  a \sigma  \sqrt{\frac{k+m}{n}}
\end{equation}
with $a\ge 120$. There exist an absolute constant $C_3>0$ and a constant $C_4>0$ depending only on a such that the following holds.
If $ n\geq C_3  s \max(k,m)$, then for
the estimator $ \hat A$ defined in \eqref{eq:def-nuclear}  with tuning parameter  \eqref{eq:trace_param1} we have that, with probability at least $1-\delta- e^{-n/C_4}$,
\begin{eqnarray*}
\|\design \hat A-\vf\|_n^2&\le&\min_{A \in \R^{k\times m}: {\rm rank}(A)\leq s }\Big(\|\design A-\vf\|_n^2+ C_4\frac{\sigma^2{\rm rank}(A)(k+m)}{n} \Big)
\\ &&
\qquad +\ 
 \frac{16\sigma^2(\Phi^{-1}(1-\delta))^2}{n}.
\end{eqnarray*}
If, in addition, $\vf =\design A^* $ for some $A^* \in \R^{k\times m}$ such that ${\rm rank}(A^*)\leq s$, then with probability at least $1-\delta-e^{-n/C_4}$,
\begin{equation*}\label{eq:soitrace-2}
\| \hat A- A^*\|_*\le C_4 \sigma \left(s\sqrt{\frac{k+m}{n}} +\frac{ (\Phi^{-1}(1-\delta))^2}{\sqrt{(k+m)n}} \right).   
\end{equation*}
\end{theorem}
\begin{proof} Under the assumptions of the theorem, $\mathbb X$ is a nearly isometric linear map, cf. \cite{MR2680543}.  Then, it follows from \cite[Lemma 4.3]{MR2680543} that there exists an absolute constant $C_5>0$ such that $\phi_{max}\le 3/2$ with probability at least $1-e^{-n/C_5}$. Therefore, we can use Theorem~\ref{theo:trace} with $\lambda$ defined in \eqref{eq:trace_param1}. Plugging the bound on $\mu_4$ from Proposition~\ref{prop:constant_mu_nuclear}  in the oracle inequalities in deviation from Theorem~\ref{theo:trace} we obtain the result.
\end{proof}

 \begin{footnotesize}
\bibliographystyle{plain}
\bibliography{convex_SMF_bib}
\end{footnotesize}

\end{document}